\documentclass[reqno]{amsart}
\usepackage[utf8]{inputenc}

\usepackage{mathrsfs}

\usepackage{amsmath,amssymb}
\usepackage{amsfonts}
\usepackage{amsthm}
\usepackage{graphicx}
\allowdisplaybreaks

\usepackage{hyperref}
\hypersetup{
    colorlinks=true,
    linkcolor=blue,
    filecolor=magenta,      
    urlcolor=cyan
}

\usepackage[normalem]{ulem}
\usepackage{cancel}

\newtheorem{theorem}{Theorem}

\newtheorem{proposition}[theorem]{Proposition}

\theoremstyle{definition}

\newtheorem{remark}[theorem]{Remark}

\newtheorem{example}[theorem]{Example}

\renewcommand{\frak}{\mathfrak}

\title[Left-invariant nearly pseudo-K\"{a}hler$\dots$]{Left-invariant nearly pseudo-K\"{a}hler structures and the tangent Lie group}

\author{David N Pham}
\address{Queensborough C. College City University of New York}
\curraddr{}
\email{dnpham@qcc.cuny.edu}



\date{}

\begin{document}
\maketitle

\begin{abstract}
    Let $G$ be a Lie group and let $(g,J)$ be a left-invariant almost pseudo-Hermitian structure on $G$. It is shown that if $(g,J)$ is also nearly pseudo-K\"{a}hler, then the tangent bundle $TG$ (with its natural Lie group structure induced from $G$) admits a left-invariant nearly pseudo-K\"{a}hler structure.  
\end{abstract}
$~$\\
Let $(M,g,J)$ be an almost Hermitian manifold and let $\nabla$ be the Levi-Civita connection associated to the metric $g$.  It is well known that $(M,g,J)$ is K\"{a}hler if and only if $\nabla J\equiv 0$.  If, instead, $(g,J)$ satisfies the weaker condition $(\nabla_XJ)X=0$ for all vector fields $X$ on $M$, then $(M,g,J)$ is called a \textit{nearly K\"{a}hler} manifold.  The most celebrated example of a nearly K\"{a}hler manifold is the $6$-sphere $S^6$, where the almost complex structure is the one derived from the octonions and the metric is just the standard round metric (cf \cite{ABF2017} and the references therein).  If one drops the condition that $g$ be positive definite and simply require $g$ to be nondegenerate (in other words, $g$ is pseudo-Riemannian), $(M,g,J)$ is called an almost complex \textit{pseudo-Hermitian} manifold.  Moreover, if $(M,g,J)$ also satisfies $(\nabla_XJ)X=0$ for all vector fields $X$, then $(M,g,J)$ is called a \textit{nearly pseudo-K\"{a}hler manifold}. In this brief note, we prove the following result:
\begin{theorem}
\label{MainTheorem}
Let $G$ be a Lie group and let $(g,J)$ be a left-invariant almost pseudo-Hermitian structure on $G$.  If $(G,g,J)$ is also nearly pseudo-K\"{a}hler, then the tangent bundle $TG$ (equipped with its natural Lie group structure induced from $G$) admits a left-invariant nearly pseudo-K\"{a}hler structure.  Moreover, if the almost complex structure $J$ on $G$ is non-integrable, then so is the induced almost complex structure on $TG$. Hence, if $(G,g,J)$ is strictly nearly pseudo-K\"{a}hler, then so is the induced structure on $TG$.
\end{theorem}

Of course, the primary upshot of Theorem \ref{MainTheorem} is that by starting with a Lie group $G$ with an existing left-invariant (strictly) nearly pseudo-K\"{a}hler structure, one obtains left-invariant (strictly) nearly pseudo-K\"{a}hler structures on the tangent Lie groups $T^kG:=T(T^{k-1}G)$ where $T^0G:=G$ for all $k\ge 1$. The problem of constructing a left-invariant nearly pseudo-K\"{a}hler structure on a given Lie group by direct calculation is algebraically daunting. Consequently, having a means of generating Lie groups with left-invariant nearly pseudo-K\"{a}hler structures is certainly advantageous.  This is the motivation for Theorem \ref{MainTheorem}.    

Since we are constructing left-invariant structures, we work entirely at the Lie algebra level.  For a Lie group $G$, we set $\frak{g}:=T_eG$ where $e$ denotes the identity element of $G$ and we identify $X\in \frak{g}$ with the left-invariant vector field whose value at $e$ is $X$ for the Lie algebra structure on $\frak{g}$.  We recall the following example from \cite{SSH2010, GV2020}:
\begin{example}
Let $SL(2,\mathbb{R})$ denote the Lie group of  $2\times 2$ real matrices with determinant $1$. As a vector space, its Lie algebra  $\frak{sl}(2,\mathbb{R})$ consists of all $2\times 2$ real traceless matrices.  A basis for $\frak{sl}(2,\mathbb{R})$ is then 
$$
X_1:=\left(\begin{array}{cc}
1 & 0\\
0 & -1
\end{array}
\right),\hspace*{0.1in} X_2:=\left(\begin{array}{cc}
0 & 1\\
1 & 0
\end{array}
\right),\hspace*{0.1in} X_3:=\left(\begin{array}{cc}
0 & 1\\
-1 & 0
\end{array}
\right).
$$
The nonzero commutator relations are 
$$
[X_1,X_2] = 2X_3,\hspace*{0.1in} [X_1,X_3]=2X_2,\hspace*{0.1in} [X_2,X_3]=-2X_1.
$$
Consider the product Lie group $SL(2,\mathbb{R})\times SL(2,\mathbb{R})$.  Its Lie algebra is the direct sum Lie algebra $\frak{sl}(2,\mathbb{R})\oplus \frak{sl}(2,\mathbb{R})$. Let 
$$
E_i=(X_i,0),\hspace*{0.1in} F_i=(0,X_i),\hspace*{0.1in} \mbox{for }i=1,2,3.
$$
A basis for the Lie algebra $\frak{sl}(2,\mathbb{R})\oplus \frak{sl}(2,\mathbb{R})$ is then 
$$
E_1,~E_2,~E_3,~F_1~,F_2,~F_3.
$$
Let $J: T(SL(2,\mathbb{R})\times SL(2,\mathbb{R}))\longrightarrow T(SL(2,\mathbb{R})\times SL(2,\mathbb{R}))$ be the unique left-invariant bundle map defined by the conditions:
$$
JE_i=\frac{1}{\sqrt{3}}(2F_i+E_i),\hspace*{0.1in} JF_i=-\frac{1}{\sqrt{3}}(2E_i+F_i).
$$
One verifies that $J^2=-id$.  Also, by calculating the Nijenhuis tensor for $J$, one sees that $J$ is also non-integrable.  Let $g$ be the unique left-invariant pseudo-Riemannian metric on $SL(2,\mathbb{R})\times SL(2,\mathbb{R})$ defined by the conditions
$$
g(E_i,F_j)=\left\{\begin{array}{ll}
-\frac{1}{3}\delta_{ij} & \mbox{for }i=1,2\\
\frac{1}{3}\delta_{ij}& \mbox{for }i=3
\end{array}\right.
$$
$$
g(E_i,E_j)=g(F_i,F_j)=\left\{\begin{array}{ll}
\frac{2}{3}\delta_{ij} & \mbox{for }i=1,2\\
-\frac{2}{3}\delta_{ij}& \mbox{for }i=3
\end{array}\right.
$$
where $\delta_{ij}=1$ for $i=j$ and zero otherwise.  One verifies that $(g,J)$ is a left-invariant (strictly) nearly pseudo-K\"{a}hler structure on $SL(2,\mathbb{R})\times SL(2,\mathbb{R})$.
\end{example}
We now take a moment to recall the natural Lie group structure on the tangent bundle $TG$ for a Lie group $G$.  As a vector bundle, $TG$ is naturally isomorphic to the vector bundle $G\times \frak{g}\rightarrow G$ via the isomorphism
$$
\varphi: TG\longrightarrow G\times \frak{g},\hspace*{0.1in} T_g G\ni A \mapsto (g, (l_{g^{-1}})_\ast A)
$$
where $l_g: G\rightarrow G$ denotes left translation by $g\in G$ and $(l_g)_\ast$ is the associated pushforward map on the tangent space.  Using the above isomorphism, we identify $TG$ with the product $G\times \frak{g}$.  However, the Lie group structure on $TG$ induced by $G$ is \textit{not} the product Lie group structure.  For $(g,X),~(h,Y)\in TG$, the group product is given by 
$$
(g,X)\cdot (h,Y):=(gh, \mbox{Ad}_{h^{-1}}X+Y)
$$
with inverse 
$$
(g,X)^{-1}:=(g^{-1},-\mbox{Ad}_gX)
$$
where $\mbox{Ad}$ is the adjoint action of $G$ on $\frak{g}$ (see \cite{PY2022} for additional details). $TG$ with the above group structure is called the \textit{tangent Lie group} associated to $G$. As a Lie group, $TG$ is isomorphic to the semi-direct product $G\ltimes \frak{g}$ where $G$ acts on $\frak{g}$ from the \textit{right} by $X\cdot g:=\mbox{Ad}_{g^{-1}}X$ and $\frak{g}$ is regarded as an abelian Lie group.  As a vector space, the Lie algebra of $TG$ is naturally isomorphic to $\frak{g}\oplus \frak{g}$. Since $TG\simeq G\ltimes \frak{g}$, the Lie algebra structure is \textit{not} the direct sum Lie algebra structure.  For $X\in \frak{g}$, define  
$$
X^c:=(X,0),\hspace*{0.1in} X^v:=(0,X).
$$
One can verify that the bracket relations on the Lie algebra of $TG$ are given by
\begin{equation}
\label{LiftBracket}
 [X^c,Y^c]=[X,Y]^c,\hspace*{0.1in} [X^c,Y^v]=[X,Y]^v, \hspace*{0.1in} [X^v,Y^v]=0.   
\end{equation}
(See \cite{PY2022,Ph2019} for additional details).  One can show that the left-invariant vector fields associated to $X^c$ and $X^v$ on $TG$ are respectively the complete and vertical lifts of the left-invariant vector field associated to $X$ on $G$ (cf \cite{YI1973}).

For an almost pseudo-Hermitian manifold $(M,g,J)$, we recall that the associated 2-form $\omega\in \Omega^2(M)$ is given by
$$
\omega(\cdot,\cdot):=g(J\cdot,\cdot). 
$$
Note that since $J$ is compatible with $g$, $J$ is also compatible with $\omega$, that is, $\omega(J\cdot,J\cdot)=\omega(\cdot, \cdot)$.  In addition, we also recall the definition of the Nijenhuis tensor $N_J$ associated to an almost complex structure $J$:
$$
N_J(X,Y):=J[JX,Y]+J[X,JY]+[X,Y]-[JX,JY]
$$
where $X$ and $Y$ are vector fields on $M$.  By the Newlander-Nirenberg Theorem, $J$ is integrable precisely when $N_J\equiv 0$.

We will find it useful to express the nearly K\"{a}hler condition $(\nabla_XJ)X=0$ explicitly in terms of $\omega$ and $J$\footnote{One can also arrive at Proposition \ref{NKCondition} by recognizing that the nearly K\"{a}hler condition is equivalent to the condition $(\nabla_X\omega)(X,Y)=0$ for all vector fields $X,Y$ on $M$ and then applying equation (2.7) of Theorem 2.1 of \cite{Gray1966}.}: 
\begin{proposition}
\label{NKCondition}
Let $(M,g,J)$ be an almost pseudo-Hermitian manifold.  Then $(M,g,J)$ is nearly pseudo-K\"{a}hler if and only if 
$$
\omega(N_J(X,Y),X)+d\omega(X,JY,JX)=0
$$
for all vector fields $X$, $Y$ on $M$.
\end{proposition}
\begin{proof}
The Levi-Civita connection associated to $g$ is uniquely defined by the Kozul formula:
\begin{align*}
    2g(\nabla_XY,Z)&=Xg(Y,Z)+Yg(Z,X)-Zg(X,Y)\\
    &+g([X,Y],Z)-g([X,Z],Y)-g([Y,Z],X).
\end{align*}
The condition $(\nabla_XJ)X=0$ for all vector fields $X$ on $M$ is equivalent to the condition that 
$$
g((\nabla_XJ)X,Y)=0
$$
for all vector fields $X$, $Y$ on $M$. Using the Kozul formula, we have
\begin{align*}
    2g((\nabla_XJ)X,Y)&=2g(\nabla_X(JX),Y)-2g(J\nabla_XX,Y)\\
    &=2g(\nabla_X(JX),Y)+2g(\nabla_XX,JY)\\
    &=Xg(JX,Y)+(JX)g(Y,X)-Yg(X,JX)\\
    &+g([X,JX],Y)-g([X,Y],JX)-g([JX,Y],X)\\
    &+2Xg(X,JY)-(JY)g(X,X)-2g([X,JY],X).
\end{align*}
From the definition of $\omega$, we have $g(\cdot,\cdot)=\omega(\cdot, J\cdot)$. Using this relation along with the compatibility of $J$ and $\omega$, the last equality can be expressed as 
\begin{align*}
    2g((\nabla_XJ)X,Y)&=-(JX)\omega(JY,X)+\omega([X,JX],JY)+\omega([X,Y],X)\\
    &+\omega(J[JX,Y],X)
    -X\omega(X,Y)-(JY)\omega(X,JX)\\
    &+\omega(J[X,JY],X)-\omega([X,JY],JX).
\end{align*}
Using the definition of the Nijenhuis tensor $N_J$ and the fact that 
\begin{align*}
    d\omega(X,Y,Z)&=X\omega(Y,Z)+Y\omega(Z,X)+Z\omega(X,Y)\\
&-\omega([X,Y],Z)-\omega([Y,Z],X)-\omega([Z,X],Y),
\end{align*}
the previous equality becomes
\begin{align*}
    2g((\nabla_XJ)X,Y)&=\omega(N_J(X,Y),X)+\omega([JX,JY],X)\\
    &-(JX)\omega(JY,X)+\omega([X,JX],JY)\\
    &-X\omega(X,Y)-(JY)\omega(X,JX)-\omega([X,JY],JX)\\
    &=\omega(N_J(X,Y),X)-\omega([JY,JX],X)\\
    &+(JX)\omega(X,JY)-\omega([JX,X],JY)\\
    &+X\omega(JY,JX)+(JY)\omega(JX,X)-\omega([X,JY],JX)\\
    &=\omega(N_J(X,Y),X)+d\omega(X,JY,JX),
\end{align*}
which completes the proof.
\end{proof}
\begin{remark}
From Proposition \ref{NKCondition}, one can deduce that if $(M,g,J)$ is nearly pseudo-K\"{a}hler with 2-form $\omega(\cdot,\cdot):=g(J\cdot, \cdot)$ closed, then $(M,g,J)$ must be pseudo-K\"{a}hler. Here, pseudo-K\"{a}hler means that $J$ is also integrable and $\omega$ is closed. Hence, a pseudo-K\"{a}hler manifold satisfies all the conditions of a K\"{a}hler manifold except the metric is not required to be positive definite.

In other words, an almost pseudo-Hermitian manifold $(M,g,J)$ which is both almost pseudo-K\"{a}hler (i.e. $\omega$ is closed, but $J$ is not necessarily integrable) and nearly pseudo-K\"{a}hler must be pseudo-K\"{a}hler.  This fact implies that  any nearly pseudo-K\"{a}hler manifold of dimension 2 or 4 must be pseudo-K\"{a}hler.  Hence, any strictly nearly pseudo-K\"{a}hler manifold must be of dimension $6$ or greater.
\end{remark}
\noindent We are now in a position to prove Theorem \ref{MainTheorem}:
\begin{proof}
Let $G$ be a Lie group with a left-invariant nearly pseudo-K\"{a}hler structure $(g,J)$.  Also, let $\mbox{Lie}(TG)$ denote the Lie algebra of the tangent Lie group $TG$.  Let $\widehat{J}: T(TG)\longrightarrow T(TG)$ be the left-invariant bundle map uniquely defined by the conditions:
$$
\widehat{J}X^c:= (JX)^c,\hspace*{0.2in} \widehat{J}X^v:=(JX)^v
$$
for all $X\in \frak{g}$.  It follows immediately from the definition that $\widehat{J}^{~2}=-id$.  Let $\widehat{g}$ be the left-invariant symmetric bilinear form on the tangent Lie group $TG$ which is uniquely defined by the conditions:
$$
\widehat{g}(X^c,Y^c)=\widehat{g}(X^v,Y^v)=0
$$
$$
\widehat{g}(X^c,Y^v):=g(X,Y)
$$
for all $X,Y\in \frak{g}$. From the definition, it immediately follows that 
$$
\widehat{g}(\widehat{J}X^c,\widehat{J}Y^c)=\widehat{g}((JX)^c,(JY)^c)=0=\widehat{g}(X^c,Y^c).
$$
Likewise, 
$$
\widehat{g}(\widehat{J}X^v,\widehat{J}Y^v)=0=\widehat{g}(X^v,Y^v).
$$
Also, 
\begin{align*}
    \widehat{g}(\widehat{J}X^c,\widehat{J}Y^v)=\widehat{g}((JX)^c,(JY)^v)=g(JX,JY)=g(X,Y)=\widehat{g}(X^c,Y^v).
\end{align*}
Hence, $\widehat{J}$ and $\widehat{g}$
are compatible.  For the nondegeneracy of $\widehat{g}$, let $\mathbf{A}$ be a nonzero element of $\mbox{Lie}(TG)$.  Then $\mathbf{A}=X^c+Y^v$ for some $X,Y\in \frak{g}$. If $X$ is nonzero, there exists a $W\in \frak{g}$ such that $g(X,W)\neq 0$.  Hence, 
$$
\widehat{g}(\mathbf{A},W^v)=\widehat{g}(X^c,W^v)=g(X,W)\neq 0.
$$
If $X=0$, then $Y\neq 0$ and by the nondegeneracy of $g$ there exists a $Z\in \frak{g}$ such that $g(Y,Z)\neq 0$.  Hence, 
$$
\widehat{g}(\mathbf{A},Z^c)=\widehat{g}(Y^v,Z^c)=g(Y,Z)\neq 0.
$$
This shows that $(\widehat{g}, \widehat{J})$ is a left-invariant pseudo-Hermitian structure on $TG$.

We now show that $(\widehat{g}, \widehat{J})$ is also nearly pseudo-K\"{a}hler.  For this, we first note that the definition of $\widehat{J}$ along with the bracket relations for complete and vertical lifts (\ref{LiftBracket}) gives the following:
\begin{equation}
    \label{N1}
    N_{\widehat{J}}(X^c,Y^c)=\left(N_J(X,Y)\right)^c
\end{equation}
\begin{equation}
    \label{N2}
    N_{\widehat{J}}(X^v,Y^v)=0
\end{equation}
\begin{equation}
    \label{N3}
    N_{\widehat{J}}(X^c,Y^v)=\left(N_J(X,Y)\right)^v
\end{equation}
for all $X,Y\in \frak{g}$.  Note that equations (\ref{N1}) and (\ref{N3}) imply that if $J$ is non-integrable, then so is $\widehat{J}$.

Let $\omega(\cdot, \cdot):=g(J\cdot, \cdot)$ and $\widehat{\omega}(\cdot, \cdot):=\widehat{g}(\widehat{J}\cdot, \cdot)$ be the 2-forms associated to $(g,J)$ and $(\widehat{g},\widehat{J})$ respectively.  Then 
\begin{equation}
    \label{w1}
    \widehat{\omega}(X^c,Y^c)=\widehat{\omega}(X^v,Y^v)=0
\end{equation}
\begin{equation}
    \label{w2}
    \widehat{\omega}(X^c,Y^v)=\omega(X,Y)
\end{equation}
for all $X,Y\in \frak{g}$.  From this, we have
\begin{equation}
    \label{dw1}
    d\widehat{\omega}(X^c,Y^c,Z^c)=d\widehat{\omega}(X^v,Y^v,Z^v)=d\widehat{\omega}(X^c,Y^v,Z^v)=0
\end{equation}
\begin{equation}
    \label{dw2}
    d\widehat{\omega}(X^c,Y^c,Z^v)=d\omega(X,Y,Z)
\end{equation}
for all $X,Y,Z\in \frak{g}$.  Now let 
$$
\psi(X,Y,Z):=\omega(N_J(X,Y),Z)+d\omega(X,JY,JZ)
$$
for $X,Y,Z\in \frak{g}$ and let
$$
\widehat{\psi}(\mathbf{A},\mathbf{B},\mathbf{C}):=\widehat{\omega}(N_{\widehat{J}}(\mathbf{A},\mathbf{B}),\mathbf{C})+d\widehat{\omega}(\mathbf{A},\widehat{J}\mathbf{B},\widehat{J}\mathbf{C})
$$
for $\mathbf{A},\mathbf{B},\mathbf{C}\in \mbox{Lie}(TG)$.  Using (\ref{N1})-(\ref{N3}) and (\ref{dw1})-(\ref{dw2}), we see that 
\begin{align}
    \label{psi1}
    \widehat{\psi}(X^c,Y^c,Z^c)&=0\\
    \label{psi2}
    \widehat{\psi}(X^c,Y^c,Z^v)&=\psi(X,Y,Z)\\
    \label{psi3}
    \widehat{\psi}(X^c,Y^v,Z^c)&=\psi(X,Y,Z)\\
    \label{psi4}
    \widehat{\psi}(X^c,Y^v,Z^v)&=0\\
      \label{psi5}
    \widehat{\psi}(X^v,Y^c,Z^c)&=\psi(X,Y,Z)\\
      \label{psi6}
    \widehat{\psi}(X^v,Y^c,Z^v)&=0\\
      \label{psi7}
    \widehat{\psi}(X^v,Y^v,Z^c)&=0\\
      \label{psi8}
    \widehat{\psi}(X^v,Y^v,Z^v)&=0
\end{align}
for $X,Y,Z\in \frak{g}$.  Since $(g,J)$ is nearly pseudo-K\"{a}hler, we have $\psi(X,Y,X)=0$ for all $X,Y\in \frak{g}$ by Proposition \ref{NKCondition}.  Of course, this condition is equivalent to the condition 
\begin{equation}
\label{psiSkew}
\psi(X,Y,Z)+\psi(Z,Y,X)=0
\end{equation}
for all $X,Y,Z\in \frak{g}$.  (\ref{psiSkew}) along with (\ref{psi1})-(\ref{psi8}) imply 
$$
\widehat{\psi}(\mathbf{A},\mathbf{B},\mathbf{C})+\widehat{\psi}(\mathbf{C},\mathbf{B},\mathbf{A})=0
$$
for all $\mathbf{A},\mathbf{B},\mathbf{C}\in \mbox{Lie}(TG)$ which in turn implies $\widehat{\psi}(\mathbf{A},\mathbf{B},\mathbf{A})=0$ for all $\mathbf{A},\mathbf{B}\in \mbox{Lie}(TG)$.  By Proposition \ref{NKCondition}, $(\widehat{g},\widehat{J})$ is a (left-invariant) nearly pseudo-K\"{a}hler structure on $TG$.  This completes the proof.
\end{proof}


\end{document}